\documentclass[12pt,a4paper,reqno]{amsart}
\usepackage{graphicx}
\usepackage{amssymb, amsmath}
\usepackage{geometry}
\usepackage{enumerate}
\usepackage[colorlinks=true,linkcolor=blue,urlcolor=blue,citecolor=blue]{hyperref}
\usepackage{etoolbox}
\patchcmd{\thmhead}{(#3)}{#3}{}{}

\newtheorem{theorem}{Theorem}[section]
\newtheorem{definition}[theorem]{Definition}
\newtheorem{lemma}[theorem]{Lemma}
\newtheorem{corollary}[theorem]{Corollary}

\theoremstyle{definition}

\newtheorem{remark}[theorem]{Remark}

\newcommand{\1}{\mathbf{1}}
\newcommand{\N}{\mathbb{N}}
\newcommand{\C}{\mathbb{C}}
\newcommand{\T}{\mathbb{T}}

\DeclareFontFamily{U}{wncy}{}
\DeclareFontShape{U}{wncy}{m}{n}{<->wncyr10}{}
\DeclareSymbolFont{mcy}{U}{wncy}{m}{n}
\DeclareMathSymbol{\Sh}{\mathord}{mcy}{"58}

\begin{document}
\title[]{Some remarks on non-symmetric polarization}

\author{Felipe Marceca}
\address{Departamento de Matem\'{a}tica - Pab I,
	Facultad de Cs. Exactas y Naturales, Universidad de Buenos Aires,
	(1428) Buenos Aires, Argentina, and CONICET-IMAS}
\thanks{This work has been supported by CONICET-PIP 11220130100329CO, ANPCyT PICT 2015-2299, UBACyT 20020130100474BA and a CONICET doctoral fellowship.}


\begin{abstract}
Let $P:\mathbb{C}^n\rightarrow \mathbb{C}$ be an $m$-homogeneous polynomial given by
\[P(x)= \sum_{1\leq j_1\leq \ldots \leq j_m \leq n} c_{j_1 \ldots j_m} x_{j_1}\ldots x_{j_m}.\]
Defant and Schl\"uters defined a non-symmetric associated $m$-form $L_P: \left(\mathbb{C}^n \right)^m\rightarrow \mathbb{C}$ by
\[L_P \left(x^{(1)},\ldots,x^{(m)} \right)= \sum_{1\leq j_1\leq \ldots \leq j_m \leq n}  c_{j_1  \ldots j_m} x_{j_1}^{(1)}\ldots x_{j_m}^{(m)}.\]
They estimated the norm of $L_P$ on $(\mathbb{C}^n, \| \cdot\|)^m$ by  the  norm  of $P$ on $(\mathbb{C}^n, \| \cdot\|)$ 
times a $(c\log n)^{m^2}$ factor for every 1-unconditional norm $\|\cdot\|$ on $\mathbb{C}^n$.
A symmetrization procedure based on a card-shuffling algorithm which (together with Defant and Schl\"uters' argument) 
brings the constant term down to $(c m \log n)^{m-1}$ is provided. 
Regarding the lower bound, it is shown that the optimal constant is bigger than $(c \log n)^{m/2}$ when $n\gg m$.
Finally, the case of $\ell_p$-norms $\|\cdot \|_p$ with $1\leq p <2$ is addressed.
\end{abstract}

\maketitle

\thispagestyle{empty}

\section{Introduction}

Let $P:\mathbb{C}^n\rightarrow \mathbb{C}$ be an $m$-homogeneous polynomial.
It is well-known that there is a unique symmetric $m$-linear form $B:(\mathbb{C}^n)^m\rightarrow \mathbb{C}$, such that $B(x,\ldots,x)=P(x)$ for all $x\in \C$.
Moreover, the \textit{polarization formula} gives an expression for the $m$-linear form $B$ in terms of $P$ (see e.g. \cite[Section 1.1]{din}).
In fact, for every $x^{(1)},\ldots,x^{(m)}\in\C$, we have
\[B\left(x^{(1)},\ldots,x^{(m)}\right)=\frac{1}{2^m m!} \sum_{\varepsilon\in \{-1,1\}^m} P\left(\varepsilon_1 x^{(1)}+ \ldots + \varepsilon_m x^{(m)}\right).\]
It follows from this identity that
\begin{align}
\label{sim}
\sup_{\left\|x^{(k)} \right\|\leq 1} \left| B \left( x^{(1)},\ldots , x^{(m)} \right) \right| \leq e^m \sup_{\|x\|\leq 1} | P ( x ) |, 
\end{align}
for any norm $\|\cdot \|$ in $\C^n$.

In \cite{nosim}, {Defant} and {Schl\"uters} defined a non-symmetric $m$-linear form $L_P$ arising from a given $m$-homogeneous polynomial $P$.
More precisely, for an $m$-homogeneous polynomial $P:\C^n\rightarrow \C$ defined by
\[P(x)= \sum_{1\leq j_1\leq \ldots \leq j_m \leq n} c_{j_1 \ldots j_m} x_{j_1}\ldots x_{j_m},\]
its associated $m$-linear form $L_P: \left(\C^n \right)^m\rightarrow \C$ is given by
\[L_P \left(x^{(1)},\ldots,x^{(m)} \right)= \sum_{1\leq j_1\leq \ldots \leq j_m \leq n}  c_{j_1  \ldots j_m} x_{j_1}^{(1)}\ldots x_{j_m}^{(m)}.\]
Assuming unconditionality of the norm $\|\cdot \|$ in $\C^n$, {Defant} and {Schl\"uters} proved that a similar estimate as in \eqref{sim} holds for $L_P$.
Before providing further details we introduce an \textit{ad hoc} definition:

\begin{definition} 
 For $m,n \in \N$, we define $C(m,n)$ as the infimum of the constants $C>0$ such that for every $m$-homogeneous polynomial $P:\C^n\rightarrow \C$
 and every 1-unconditional norm $\|\cdot \|$ on $\C^n$ we have
 \[\sup_{\|x^{(k)} \|\leq 1} \left| L_P \left( x^{(1)},\ldots , x^{(m)} \right) \right| \leq C \sup_{\|x\|\leq 1} | P ( x ) |.\]
 Similarly, for $1\leq p<2$, we take $C_p(m,n)$ as the infimum of the constants $C>0$ such that for every $m$-homogeneous polynomial $P:\C^n\rightarrow \C$ we have
 \[\sup_{\|x^{(k)} \|_p\leq 1} \left| L_P \left( x^{(1)},\ldots , x^{(m)} \right) \right| \leq C \sup_{\|x\|_p\leq 1} | P ( x ) |.\]
\end{definition}

The aforementioned result of \cite{nosim} can be stated in terms of the previous definition.

\begin{theorem}[{\cite[Theorem 1.1]{nosim}}]
\label{teonosim}
There exists a universal constant $c_1\geq 1$ such that 
\[C(m,n)\leq (c_1 \log n)^{m^2}.\]
Moreover, for $1\leq p<2$, there is a constant $c_2=c_2(p)\geq 1$ for which
\[C_p(m,n) \leq c_2^{m^2}.\]
\end{theorem}

Note that by the uniqueness of the symmetric $m$-linear form $B$ we have
\begin{align}
\label{sim2}
 B \left(x^{(1)},\ldots,x^{(m)} \right)=\frac{1}{m!}\sum_{\sigma\in \Sigma_m} L_P \left(x^{\sigma(1)},\ldots,x^{\sigma(m)} \right),
\end{align}
where $\Sigma_m$ is the group of permutations of $m$ elements.
The proof of Theorem \ref{teonosim} consists of bounding the norm of $L_P$ by successive partial symmetrizations starting at $L_P$
and ending at the fully symmetrized $B$. Finally, applying \eqref{sim} yields the result.
Changing only the way in which this symmetrization is carried out and using the same arguments as in \cite{nosim},
we obtain improved bounds for the constants $C(m,n)$ and $C_p(m,n)$. Additionally, we provide lower bounds for these constants. Our main result is the following.

\begin{theorem}
\label{teo1}
There exists a universal constant $c_1\geq 1$ such that 
\[\left(\frac{\log\left(\frac{2n}{m}\right)-\pi}{\pi}\right)^{m/2} \leq C(m,n)\leq c_1^m m^m (\log n)^{m-1}.\]
Moreover, for $1\leq p<2$, there is a constant $c_2=c_2(p)\geq 1$ for which 
\[m^{\frac{m}{p}}\leq C_p(m,n) \leq c_2^m m^m.\]
\end{theorem}


\begin{remark}
 {Defant} and {Schl\"uters} achieved similar upper bounds by refining their original calculations from \cite{nosim} as it was mentioned during a personal communication.
\end{remark}

\begin{remark}
\label{rem1}
Scrutiny of the theorem's proof suggests that the underlying reason which determines the magnitude of the constants $C(m,n)$ and $C_p(m,n)$ 
is the behaviour of the operator known as the main triangle projection.
Roughly speaking, the main triangle projection is the operator which given a matrix in $\C^{n\times n}$ returns the same matrix with zeroes below the diagonal.
Each norm on $\C^n$ induces an operator norm in $\C^{n\times n}$ and again this induces a norm for the main triangle projection.
Estimations of the latter norm are the ones that shape the upper and lower bounds of $C(m,n)$ and $C_p(m,n)$ that were obtained.
\end{remark}

\section{Symmetrization}
The following may be deduced from \eqref{sim2}.
\begin{align*}
 B \left(x^{(1)},\ldots,x^{(m)} \right)&=\frac{1}{m!}\sum_{\sigma\in \Sigma_m} L_P \left(x^{\sigma(1)},\ldots,x^{\sigma(m)} \right)
 \\ &=\frac{1}{m!}\sum_{\sigma\in \Sigma_m} \sum_{1\leq j_1\leq \ldots \leq j_m \leq n}  c_{j_1  \ldots j_m} x_{j_1}^{\sigma(1)}\ldots x_{j_m}^{\sigma(m)}
 \\ &=\frac{1}{m!}\sum_{\sigma\in \Sigma_m} \sum_{1\leq j_1\leq \ldots \leq j_m \leq n}  c_{j_1  \ldots j_m} x_{j_{\sigma^{-1}(1)}}^{(1)}\ldots x_{j_{\sigma^{-1}(m)}}^{(m)}
 \\ &=\frac{1}{m!}\sum_{\tau\in \Sigma_m} \sum_{1\leq j_1\leq \ldots \leq j_m \leq n}  c_{j_1  \ldots j_m} x_{j_{\tau(1)}}^{(1)}\ldots x_{j_{\tau(m)}}^{(m)}.
\end{align*}

From a probabilistic point of view, this may be restated as
\begin{align}
\label{sim3}
B \left(x^{(1)},\ldots,x^{(m)} \right)
=E \left[ \sum_{1\leq j_1\leq \ldots \leq j_m \leq n}  c_{j_1  \ldots j_m} x_{j_{\sigma(1)}}^{(1)}\ldots x_{j_{\sigma(m)}}^{(m)}\right], 
\end{align}
where expectation is taken over $\sigma\in \Sigma_m$ and $\Sigma_m$ is endowed with the equiprobability measure. 
In other words, $B$ is the expected value of $L_P$ when the order of the monomials' subindices is an equidistributed random variable. 
Thus, a card-shuffling procedure applied to the order of the subindices will yield a symmetrization procedure for $L_P$ by taking expectation.
We will use the {Fischer-Yates} shuffle in its original version which can be found in \cite{fiya}. It goes as follows.
Choose a random card from an ordered deck and leave it on top.
Next, choose a random card between the second and the last place and leave it in the second place, and so on.
At the last step, choose between the last two cards which one will go in the penultimate place.
After applying this procedure, an ordered deck will be completely shuffled, that is, any arrangement will be equally probable.

\begin{remark}
 Note that at any given step, the $k$-th step say, the first $k-1$ cards (which have been previously selected) are completely random,
while the last cards remain completely ordered. This special structure will be crucial in the proof of Theorem \ref{teo1}.
\end{remark}

Next, we introduce the symmetrization procedure arising from the {Fischer-Yates} shuffle. 
For every $1\leq k \leq m-1$ we let $\mathbb{P}_k$ be the probability distribution on $\Sigma_m$ associated to performing the first $k$ steps of the shuffling algorithm. 
We define the $k$-th shuffle $S_k$ of an $m$-form $L: \left(\C^n \right)^m\rightarrow \mathbb{C}$ by
\[S_k L\left(x^{(1)},\ldots,x^{(m)} \right) 
  =E \left[ \sum_{i_1,\ldots, i_m=1 }^n  c_{i_1  \ldots i_m} x_{i_{\sigma(1)}}^{(1)}\ldots x_{i_{\sigma(m)}}^{(m)}\right],\]
where $\sigma \sim \mathbb{P}_k.$

In particular, from \eqref{sim3} and the fact that the $(m-1)$-th step of the shuffle achieves equidistribution we have
\[B= S_{m-1} L_P.\] However, it should be noticed that the intermediate shuffles are not partial symmetrizations 
since we are symmetrizing the monomials' subindices rather than the variables.

In order to study the structure of $S_k$, we define the $k$-th shuffling step $T_k$ of an $m$-form $L: \left(\mathbb{C}^n \right)^m\rightarrow \mathbb{C}$ by
\begin{multline*}
 T_k L  \left(x^{(1)},\ldots,x^{(m)} \right)
 \\ = \frac{1}{m-k+1} \sum_{l=k}^m L \left(x^{(1)},\ldots,x^{(k-1)},x^{(k+1)},\ldots,x^{(l)},x^{(k)},x^{(l+1)},\ldots,x^{(m)} \right). 
\end{multline*}

\begin{lemma}
 For every $1\leq k \leq m-1$ we have  that $S_k= T_k \ldots T_1$.
\end{lemma}
\begin{proof}
 Since $T_k$ and $S_k$ are linear for every $1\leq k \leq m-1$, it is enough to check that the equality holds for monomials.
 Fix $1\leq i_1, \ldots, i_m \leq n$, we have to prove that
 \[S_k \left(x^{(1)}_{i_1} \ldots x^{(m)}_{i_m}\right)= T_k \ldots T_1\left(x^{(1)}_{i_1} \ldots x^{(m)}_{i_m}\right).\]
 We will proceed by induction.
 If $k=1$, the random permutation $\sigma$ is a cycle in $\Sigma_m$. 
 More precisely, using the cycle notation in $\Sigma_m$ we have that $\sigma$ takes the value $(l \ l-1 \ \ldots \ 1)$ for some $1\leq l \leq m$ with probability $1/m$.
 Therefore, we get
 \begin{align*}
  S_1\left(x^{(1)}_{i_1} \ldots x^{(m)}_{i_m}\right)&=E \left[ x_{i_{\sigma(1)}}^{(1)}\ldots x_{i_{\sigma(m)}}^{(m)}\right]
  =\frac{1}{m} \sum_{l=1}^m x^{(1)}_{i_{l}} x^{(2)}_{i_{1}} \ldots  x^{(l)}_{i_{l-1}} x^{(l+1)}_{i_{l+1}} \ldots x^{(m)}_{i_{m}}
  \\ &=\frac{1}{m} \sum_{l=1}^m  x^{(2)}_{i_{1}} \ldots  x^{(l)}_{i_{l-1}} x^{(1)}_{i_{l}} x^{(l+1)}_{i_{l+1}} \ldots x^{(m)}_{i_{m}} 
  =T_1\left(x^{(1)}_{i_1} \ldots x^{(m)}_{i_m}\right).
 \end{align*}
 
 Only the inductive step remains to be proven. Let $2\leq k \leq m-1$ and suppose the lemma holds for $k-1$. 
 From the definition of the {Fischer-Yates} shuffle we may deduce that a random permutation with law $\mathbb{P}_k$ can be written 
 as the composition of two independent random permutations $\tau$ and $\sigma$ where $\sigma \sim \mathbb{P}_{k-1}$ 
 and $\tau$ takes the value $\tau_l =(l \ l-1 \ \ldots \ k)$ for some $k\leq l \leq m$ with probability $1/{(m-k+1)}$.
For a fixed $\tau$, we may define new indices $j_1,\ldots,j_m$ such that $j_k=i_{\tau(k)}$ for every $1\leq k\leq m$. 
So we obtain
\begin{align*}
S_k \left(x^{(1)}_{i_1} \ldots \right.& \left. x^{(m)}_{i_m} \right) 
=E_{\tau,\sigma} \left[  x_{i_{\tau\sigma(1)}}^{(1)}\ldots x_{i_{\tau\sigma(m)}}^{(m)}\right]
=E_\tau \left[E_\sigma \left[  x_{j_{\sigma(1)}}^{(1)}\ldots x_{j_{\sigma(m)}}^{(m)}\right]\right]
\\&=E_\tau \left[S_{k-1} \left(x^{(1)}_{j_1} \ldots x^{(m)}_{j_m} \right)\right]
=E_\tau \left[S_{k-1} \left(x_{i_{\tau(1)}}^{(1)}\ldots x_{i_{\tau(m)}}^{(m)} \right)\right]
\\&=\frac{1}{m-k+1} \sum_{l=k}^m S_{k-1} \left(x^{(1)}_{i_{\tau_l(1)}} \ldots x^{(m)}_{i_{\tau_l(m)}} \right)
\\&=\frac{1}{m-k+1} \sum_{l=k}^m S_{k-1} 
\left(x^{(1)}_{i_1} \ldots x^{(k-1)}_{i_{k-1}} x^{(k)}_{i_{l}} x^{(k+1)}_{i_{k}} \ldots x^{(l)}_{i_{l-1}} x^{(l+1)}_{i_{l+1}} \ldots x^{(m)}_{i_m} \right)
\\&=\frac{1}{m-k+1} \sum_{l=k}^m S_{k-1}
\left(x^{(1)}_{i_1} \ldots x^{(k-1)}_{i_{k-1}}  x^{(k+1)}_{i_{k}} \ldots x^{(l)}_{i_{l-1}} x^{(k)}_{i_{l}} x^{(l+1)}_{i_{l+1}} \ldots x^{(m)}_{i_m} \right)
\\&=T_k S_{k-1} \left(x^{(1)}_{i_1} \ldots  x^{(m)}_{i_m} \right) 
,
\end{align*}
which completes the proof.
\end{proof}



Following \cite{nosim}, we turn to study how the coefficients of the succesive shuffles of $L_P$ change.
Let $L: \left(\C^n \right)^m\rightarrow \mathbb{C}$ be an $m$-linear form given by
\[L \left(x^{(1)},\ldots,x^{(m)} \right)=\sum_{i \in \mathcal{I}(m,n)} c_i x_{i_1}^{(1)} \ldots x_{i_m}^{(m)},\]
where $\mathcal{I}(m,n)=\{1,\ldots , n\}^m$. We will denote its coefficients by $c_i(L)=c_i$. 

\begin{lemma}
\label{coef2}
 For $m,n \in \N$, $1 \leq k \leq m-1$, $i\in \mathcal{I}(m,n)$ and an $m$-homogeneous polynomial $P:\mathbb{C}^n\rightarrow \mathbb{C}$ we have
  \[c_i \left(S_{k-1} L_P \right)=\begin{cases}
                  (m-k+1) \left(1+\sum_{u=1}^{m-k} \delta_{i_k,i_{k+u}}\left(\frac{1}{u+1}-\frac{1}{u}\right)\right) c_i \left(S_k L_P \right) \quad &\text{if } i_k\leq i_{k+1}
                  \\ 0 &\text{otherwise}
                 \end{cases}
,\]
where $\delta$ is the {Kronecker} delta and we take $S_0 L_P = L_P$.
\end{lemma}
\begin{proof}
 We begin the proof by calculating the coefficents $c_i\left(S_k L_P \right)$ in terms of the coefficients $c_i \left(S_{k-1} L_P \right)$.
 Observe that for an $m$-linear form $L: \left(\C^n \right)^m\rightarrow \mathbb{C}$ we have
 \begin{align*}
 T_k &L  \left(x^{(1)},\ldots,x^{(m)} \right) \notag
 \\& =\frac{1}{m-k+1} \sum_{l=k}^m L \left(x^{(1)},\ldots,x^{(k-1)},x^{(k+1)},\ldots,x^{(l)},x^{(k)},x^{(l+1)},\ldots,x^{(m)} \right) \notag
 \\& =\frac{1}{m-k+1} \sum_{l=k}^m \sum_{i \in \mathcal{I}(m,n)} c_i(L)  
 x_{i_1}^{(1)}\ldots x_{i_{k-1}}^{(k-1)}x_{i_{k}}^{(k+1)} \ldots x_{i_{l-1}}^{(l)} x_{i_{l}}^{(k)} x_{i_{l+1}}^{(l+1)} \ldots x_{i_{m}}^{(m)} \notag
 \\& = \sum_{i \in \mathcal{I}(m,n)} \frac{1}{m-k+1} \sum_{l=k}^m  c_i(L) 
 x_{i_1}^{(1)}\ldots x_{i_{k-1}}^{(k-1)} x_{i_{l}}^{(k)} x_{i_{k}}^{(k+1)} \ldots x_{i_{l-1}}^{(l)}  x_{i_{l+1}}^{(l+1)} \ldots x_{i_{m}}^{(m)} \notag
 \\& = \sum_{i \in \mathcal{I}(m,n)} \frac{1}{m-k+1} \sum_{l=k}^m  c_{\left(i_1,\ldots,i_{k-1},i_{k+1}, \ldots, i_{l},i_{k},i_{l+1}, \ldots, i_{m} \right)}(L) 
 x_{i_1}^{(1)} \ldots x_{i_{m}}^{(m)}.
 \end{align*}
 Therefore, since $S_k=T_k S_{k-1}$, we deduce the formula 
 \begin{align}
   \label{coef1}
   c_i \left(S_k L_P \right)=\frac{1}{m-k+1} \sum_{l=k}^m  c_{\left(i_1,\ldots,i_{k-1},i_{k+1}, \ldots, i_{l},i_{k},i_{l+1}, \ldots, i_{m} \right)} \left(S_{k-1} L_P \right).
 \end{align}

 By the definition of $L_P$ if a coefficient $c_i \left(L_P \right)$ is not zero, then the index $i$ must satisfy that $1\leq i_1 \leq \ldots \leq i_m \leq n$.
 We will prove inductively that for $0 \leq k \leq m-1$, if the coefficient $c_i \left(S_k L_P \right)$ is not zero,
 then the index $i$ must satisfy that $1\leq i_{k+1} \leq \ldots \leq i_m \leq n$.
 
 Since $S_0 L_P = L_P$, the case $k=0$ is already proven.
 Now assume the assertion holds for $0 \leq k-1 \leq m-1$ and fix $i\in \mathcal{I}(m,n)$ such that $i_s >i_{s+1}$ for some $k+1\leq s \leq m-1$.
 Applying the inductive hypothesis we may deduce that
 \[c_{\left(i_1,\ldots,i_{k-1},i_{k+1}, \ldots, i_{l},i_{k},i_{l+1}, \ldots, i_{m} \right)} \left(S_{k-1} L_P \right)=0,\]
 for every $k\leq l \leq m$.
 Hence, using \eqref{coef1} we get that $c_i \left(S_k L_P \right)=0$ proving the inductive step.
 In particular,  we have shown that $c_i \left(S_{k-1} L_P \right)=0$ if $i_k>i_{k+1}$ as sought. 
 
 Now assume that $i_k\leq i_{k+1}$.
 If for some $k+1\leq s \leq m-1$ we have that $i_s >i_{s+1}$, then by the previous argument we may deduce that $c_i \left(S_{k-1} L_P \right)=c_i \left(S_k L_P \right)=0$ as desired.
 Therefore, it remains to check the statement when $1\leq i_k \leq \ldots \leq i_m \leq n$. Define $s = \sup \{k \leq u \leq m \ : i_u=i_k\}$ and notice that
 \[c_{\left(i_1,\ldots,i_{k-1},i_{k+1}, \ldots, i_{l},i_{k},i_{l+1}, \ldots, i_{m} \right)} \left(S_{k-1} L_P \right)=\begin{cases}
                  c_i \left(S_{k-1} L_P \right) \quad &\text{if } k\leq l\leq s
                  \\ 0 &\text{if } s< l\leq m
                 \end{cases}
 .\]
 Thus, we may push \eqref{coef1} further to get
 \begin{align*}
  c_i \left(S_k L_P \right)&=\frac{1}{m-k+1} \sum_{l=k}^m  c_{\left(i_1,\ldots,i_{k-1},i_{k+1}, \ldots, i_{l},i_{k},i_{l+1}, \ldots, i_{m} \right)} \left(S_{k-1} L_P \right) \notag
  \\ &=\frac{1}{m-k+1} \sum_{l=k}^s  c_i \left(S_{k-1} L_P \right)=\frac{s-k+1}{m-k+1}  c_i \left(S_{k-1} L_P \right).
 \end{align*}
 Since $s\geq k$, we have that $s-k+1\neq 0$. Thus, we get
 \begin{align*}
 \label{coef3}
  c_i \left(S_{k-1} L_P \right)&=\frac{m-k+1}{s-k+1}  c_i \left(S_k L_P \right)
  \\ &=(m-k+1) \left(1+\sum_{u=1}^{s-k} \left(\frac{1}{u+1}-\frac{1}{u}\right)\right) c_i \left(S_k L_P \right)
  \\ &=(m-k+1) \left(1+\sum_{u=1}^{m-k} \delta_{i_k,i_{k+u}}\left(\frac{1}{u+1}-\frac{1}{u}\right)\right) c_i \left(S_k L_P \right).
 \end{align*}
 This concludes the proof.
\end{proof}

As in \cite{nosim}, we will restate the previous lemma using {Schur} products.
For $A,B\in\C^{\mathcal{I}(m,n)}$, the {Schur} product $A*B$ is given by
\[c_i(A*B)=c_i(A)c_i(B),\]
where $c_i(\cdot)$ denotes de $i$-th entry of a matrix.
By identifying an $m$-linear form with its coefficients, we may compute the product between a matrix and an $m$-form. 
More precisely, for $A\in\C^{\mathcal{I}(m,n)}$ and an $m$-linear form $L: \left(\mathbb{C}^n \right)^m\rightarrow \mathbb{C}$ we define
$A*L: \left(\mathbb{C}^n \right)^m\rightarrow \mathbb{C}$ by
\[c_i(A*L)=c_i(A)c_i(L).\]
With this notation Lemma \ref{coef2} proves the formula
\begin{align}
 \label{formula1}
 S_{k-1} L_P=R_k*S_k L_P,
\end{align}
where $R_k\in\C^{\mathcal{I}(m,n)}$ is given by
\begin{align*}
 c_i \left(R_k \right)=
 \begin{cases}
  (m-k+1) \left(1+\sum_{u=1}^{m-k} \delta_{i_k,i_{k+u}}\left(\frac{1}{u+1}-\frac{1}{u}\right)\right) \quad &\text{if } i_k\leq i_{k+1} 
  \\ 0 &\text{otherwise}
 \end{cases}. 
\end{align*}

The matrix $R_k\in\C^{\mathcal{I}(m,n)}$ may be decomposed as sums and products of simpler matrices. 
For $u,v\in\{1,\ldots,m\}$, let $D^{u,v},T^{u,v}\in \mathbb{C}^{\mathcal{I}(m,n)}$ be such that for every $i\in \mathcal{I}(m,n)$ we have
  \[ c_i \left(D^{u,v} \right)= \begin{cases}
                  1 \quad &\text{if } i_u= i_v
                  \\ 0 &\text{otherwise}
                 \end{cases},\]
  \[ c_i \left(T^{u,v} \right)= \begin{cases}
                  1 \quad &\text{if } i_u\leq i_v
                  \\ 0 &\text{otherwise}
                 \end{cases}.\]               
Keeping Remark \ref{rem1} in mind, we may observe that $T^{u,v}$ bears a close ressemblance with the main triangle projection $T:\C^{n\times n}\rightarrow\C^{n\times n}$.
Indeed, note that $c_i(T^{u,v})=c_{i_u,i_v}(T)$ for every $i\in\mathcal{I}(m,n)$.
\begin{lemma}
 \label{lema1}
 For $1\leq k \leq m-1$, we have
 \[R_k = (m-k+1)T^{k,k+1}*\left(1+\sum_{u=1}^{m-k} D^{k,k+u}\left(\frac{1}{u+1}-\frac{1}{u}\right)\right).\]
\end{lemma}
\begin{proof}
 For $i\in \mathcal{I}(m,n)$, we deduce that
 \begin{align*}
  c_i\left((m-k+1)\vphantom{\left(1+\sum_{u=1}^{m-k} D^{k,k+u}\left(\frac{1}{u+1}-\frac{1}{u}\right)\right)}\right.
     &\left. T^{k,k+1}*\left(1+\sum_{u=1}^{m-k} D^{k,k+u}\left(\frac{1}{u+1}-\frac{1}{u}\right)\right)\right)=
  \\ &=(m-k+1) c_i \left(T^{k,k+1} \right) \left(1+\sum_{u=1}^{m-k} c_i \left(D^{k,k+u} \right)\left(\frac{1}{u+1}-\frac{1}{u}\right)\right)
  \\ &= c_i \left(T^{k,k+1} \right) (m-k+1) \left(1+\sum_{u=1}^{m-k} \delta_{i_k,i_{k+u}}\left(\frac{1}{u+1}-\frac{1}{u}\right)\right)
  \\ &= c_i \left(R_k \right),
 \end{align*}
 which proves the statement.
\end{proof}

\section{Upper bounds}
In this section we provide the upper bounds for Theorem \ref{teo1}.
Let $\|\cdot\|$ be a norm on $\C^n$. For $A\in \C^{\mathcal{I}(m,n)}$, we define $\mu_{\|\cdot\|}(A)$ as the infimum of the constants $C>0$
such that for every $m$-linear form $L: \left(\C^n \right)^m\rightarrow \C$ we have
\[\sup_{\| x^{(k)} \|\leq 1}  \left|A*L \left(x^{(1)},\ldots,x^{(m)} \right) \right|\leq C \sup_{\| x^{(k)} \|\leq 1}  \left|L \left(x^{(1)},\ldots,x^{(m)} \right) \right|.\]
Note that $ \left(\C^{\mathcal{I}(m,n)},\mu_{\|\cdot\|} \right)$ is a {Banach} algebra.

We will use the following lemma by {Defant} and {Schl\"uters}.
\begin{lemma}[{\cite[Lemma 3.2]{nosim}}]
\label{lema2}
  For every $n,m\in\mathbb{N}$, every $u,v\in\{1,\ldots,m\}$ and every 1-unconditional norm $\| \cdot \|$ on $\mathbb{C}^n$
  \begin{align*} 
   \mu_{\| \cdot \|} \left(D^{u,v} \right)&=1,
   \\ \mu_{\| \cdot \|} \left(T^{u,v} \right)&\leq \log_2(2n).
  \end{align*}
  
 Moreover, for every $1\leq p<2$, there exists a constant $c=c(p)$ so that for every $n,m\in\mathbb{N}$
 \[\mu_{\| \cdot \|_p} \left(T^{u,v} \right)\leq c.\]
\end{lemma}
As mentioned in Remark \ref{rem1}, the estimates for $T^{u,v}$ rely on bounds for the norm of the main triangle projection
obtained by Kwapie\'n and Pe{\l}czy\'nski in \cite{kwap} and Bennett in \cite{ben}.

\begin{corollary}
 For every $n,m\in\mathbb{N}$, every $1\leq k \leq m-1$ and every 1-unconditional norm $\| \cdot \|$ on $\mathbb{C}^n$ we have
 \[\mu_{\| \cdot \|} \left(R_k \right)\leq 2(m-k+1)\mu_{\| \cdot \|} \left(T^{k,k+1} \right).\]
\end{corollary}
\begin{proof}
 From the last lemma we know that $\mu_{\| \cdot \|} \left(D^{u,v} \right)=1$ for every $u,v\in\{1,\ldots,m\}$.
 Since $ \left(\C^{\mathcal{I}(m,n)},\mu_{\|\cdot\|} \right)$ is a {Banach} algebra, we may deduce from Lemma \ref{lema1} that
 \begin{align*}
  \mu_{\| \cdot \|} \left(R_k \right)&=\mu_{\| \cdot \|}\left((m-k+1)T^{k,k+1}*\left(1+\sum_{u=1}^{m-k} D^{k,k+u}\left(\frac{1}{u+1}-\frac{1}{u}\right)\right)\right)
  \\ &\leq (m-k+1) \mu_{\| \cdot \|} \left(T^{k,k+1} \right)\left(1+\sum_{u=1}^{m-k} \mu_{\| \cdot \|} \left(D^{k,k+u} \right)\left|\frac{1}{u+1}-\frac{1}{u}\right|\right)
  \\ &\leq (m-k+1) \left(1+\sum_{u=1}^{\infty} \left(\frac{1}{u}-\frac{1}{u+1}\right)\right) \mu_{\| \cdot \|} \left(T^{k,k+1} \right)
  \\ &= 2(m-k+1) \mu_{\| \cdot \|} \left(T^{k,k+1} \right),
 \end{align*}
 as required.
\end{proof}

We are ready to prove the upper bounds for Theorem \ref{teo1}.
\begin{theorem}
 \label{teo2}
There exists a universal constant $c_1\geq 1$ such that 
 \[C(m,n)\leq c_1^m m^m (\log n)^{m-1}.\]
Moreover, for $1\leq p<2$, there is a constant $c_2=c_2(p)\geq 1$ for which 
 \[C_p(m,n) \leq c_2^m m^m.\]
\end{theorem}
\begin{proof}
 Using \eqref{formula1}, the definition of $\mu_{\| \cdot \|}$ and the previous corollary we get
 \begin{align*}
  \sup_{\| x^{(k)} \|\leq 1}  \left|S_{k-1}L_P \left(x^{(1)}\right.\right. & \left.\left. ,\ldots,x^{(m)} \right) \right| =  \sup_{\| x^{(k)} \|\leq 1}  \left|R_k*S_k L_P \left(x^{(1)},\ldots,x^{(m)} \right) \right|
  \\ &\leq \mu_{\| \cdot \|} \left(R_k \right) \sup_{\| x^{(k)} \|\leq 1}  \left|S_k L_P \left(x^{(1)},\ldots,x^{(m)} \right) \right|
  \\ &\leq 2(m-k+1) \mu_{\| \cdot \|} \left(T^{k,k+1} \right) \sup_{\| x^{(k)} \|\leq 1}  \left|S_k L_P \left(x^{(1)},\ldots,x^{(m)} \right) \right|,
 \end{align*}
 for every $1\leq k \leq m-1$. Taking $\mu=\sup_{1\leq k\leq m-1}\mu_{\| \cdot \|} \left(T^{k,k+1} \right)$ and linking the previous inequalities together, we deduce that
 \begin{align*}
  \sup_{\| x^{(k)} \|\leq 1}  \left|L_P \left(x^{(1)},\ldots,x^{(m)} \right) \right| &\leq 2m\mu \sup_{\| x^{(k)} \|\leq 1}  \left|S_1 L_P \left(x^{(1)},\ldots,x^{(m)} \right) \right|
  \\ &\leq 2^2m(m-1)\mu^2 \sup_{\| x^{(k)} \|\leq 1}  \left|S_2 L_P \left(x^{(1)},\ldots,x^{(m)} \right) \right|
  \\ &\leq \ldots \leq 2^{m-1}m!\mu^{m-1} \sup_{\| x^{(k)} \|\leq 1}  \left|S_{m-1} L_P \left(x^{(1)},\ldots,x^{(m)} \right) \right|.
 \end{align*}
 Using the identity $S_{m-1} L_P=B$ and applying \eqref{sim}, we obtain
 \begin{align*}
  \sup_{\| x^{(k)} \|\leq 1}  \left|L_P \left(x^{(1)},\ldots,x^{(m)} \right) \right| &\leq 2^{m-1}m!\mu^{m-1} \sup_{\| x^{(k)} \|\leq 1}  \left|B \left(x^{(1)},\ldots,x^{(m)} \right) \right|
  \\ &\leq 2^{m-1}e^{m}m!\mu^{m-1} \sup_{\| x \|\leq 1} |P(x)|.
 \end{align*}
 The theorem follows by applying {Stirling's} formula to estimate $m!$ and Lemma \ref{lema2} to estimate $\mu$.
\end{proof}

\section{Lower Bounds}
Firstly, we provide a lower bound for $C_p(m,n)$.
\begin{lemma}
\label{lema3}
 For every $n\geq m$ and every $1\leq p<2$, we have that $C_p(m,n)\geq m^{\frac{m}{p}}$. 
\end{lemma}
\begin{proof}
Let $P:\C^m \rightarrow \C$ be the $m$-homogeneous polynomial defined by 
\[P(x)=x_1\ldots x_m.\]
So, its associated $m$-linear form $L_P: \left(\C^m\right)^m \rightarrow \C$ is given by 
\[L_P \left(x^{(1)}, \ldots ,x^{(m)} \right)=x_1^{(1)} \ldots x_m^{(m)}.\]
Observe that
\begin{align}
\label{eq1}
\sup_{\| x^{(k)} \|_p\leq 1}  \left|L_P \left(x^{(1)},\ldots,x^{(m)} \right) \right|=\sup_{\| x^{(k)} \|_p\leq 1}  \left|x_1^{(1)} \ldots x_m^{(m)} \right|= 1. 
\end{align}
where equality is achieved by taking $x^{(i)}$ to be the $i$-th canonical vector of $\ell_p^m$.

On the other hand, a straightforward computation using {Lagrange} multipliers gives
\begin{align}
 \label{eq2}
 \sup_{\| x \|_p\leq 1} |P(x)|=  \left|P\left(m^{-\frac{1}{p}}(1,\ldots, 1) \right) \right|=m^{-\frac{m}{p}} . 
\end{align}
Applying \eqref{eq1} and \eqref{eq2} together with the definition of $C_p(m,n)$ we get
\[1=\sup_{\| x^{(k)} \|_p\leq 1}  \left|L_P \left(x^{(1)},\ldots,x^{(m)} \right) \right|\leq C_p(m,n) \sup_{\| x \|_p\leq 1} |P(x)|=m^{-\frac{m}{p}}C_p(m,n),\]
as desired.
\end{proof}

Secondly, we estimate $C(m,n)$ from below. In order to do this we will need the following special case of a theorem proved by {Pe{\l}czy\'nski}.
\begin{theorem}[{\cite[Theorem 1]{pel}}]
\label{teo3}
 For a finite index set $J$, let $(a_j)_{j\in J}$ and $(b_j)_{j\in J}$ be sequences of characters on compact abelian groups $S$ and $T$ respectively.
 Suppose there are constants $c_1,c_2>0$ such that
 \begin{align}
  \label{eq4}
  \frac{1}{c_1} \left\| \sum_{j \in J} \alpha_j a_j \right\|_{C(S)} \leq \left\| \sum_{j \in J} \alpha_j b_j \right\|_{C(T)} 
  \leq c_2 \left\| \sum_{j \in J} \alpha_j a_j \right\|_{C(S)},
 \end{align}
 for every sequence of scalars $(\alpha_j)_{j \in J}\subseteq \C$.
 Then, for every {Banach} space $E$ and every sequence of vectors $(v_j)_{j \in J}\subseteq E$ we have
 \begin{align}
  \label{eq5}
  \frac{1}{c_1 c_2} \int_S \left\| \sum_{j \in J} v_j a_j(s) \right\|_E \, ds \leq \int_T \left\| \sum_{j \in J} v_j b_j(t) \right\|_E \, dt
  \leq c_1 c_2 \int_S \left\| \sum_{j \in J} v_j a_j(s) \right\|_E \, ds.
 \end{align}
\end{theorem}

We are ready to provide the lower bound for $C(m,n)$ stated in Theorem \ref{teo1}.
\begin{lemma}
\label{lema4}
 For  $n,m\in\mathbb{N}$ such that $\log\left(\frac{2n}{m}\right)\geq\pi$, we have
 \[C(m,n)\geq \left(\frac{\log\left(\frac{2n}{m}\right)-\pi}{\pi}\right)^{m/2}.\]
\end{lemma}
\begin{proof}
 Consider the norm $\|\cdot \|_\infty$ on $\C^n$.
 Since $P(x)=L_P(x,\ldots ,x)$, we deduce that
 \[\sup_{\| x \|_\infty\leq 1} |P(x)|\leq\sup_{\| x^{(k)} \|_\infty\leq 1}  \left|L_P \left(x^{(1)},\ldots,x^{(m)}\right)\right| 
 \leq C(m,n) \sup_{\| x \|_\infty \leq 1} |P(x)|,\]
 for every $m$-homogeneous polynomial $P:\C^n\rightarrow \C$. Equivalently, by the maximum modulus principle we get
 \begin{align}
 \label{eq3}
  \sup_{x \in \T^n} |P(x)|\leq\sup_{x^{(k)}\in \T^n}  \left|L_P \left(x^{(1)},\ldots,x^{(m)}\right)\right| 
  \leq C(m,n) \sup_{x \in \T^n} |P(x)|,
 \end{align}
 where $\T=\{z\in\C \ :\ |z|=1\}$.
 
 Thus, the conditions of {Pe{\l}czy\'nski's} theorem are satisfied. Indeed, denote the compact abelian groups $\T^n$ and $\left(\T^n\right)^m$ by $S$ and $T$ respectively
 and consider the index set $J=\{j \in \mathcal{I}(m,n) \ : \ 1 \leq j_1 \leq \ldots \leq j_m \leq n\}$.
 For every $j \in J$, define the characters $a_j:S \rightarrow \T$ and $b_j: T \rightarrow \T$ by
 \[a_j(x)=x_{j_1} \ldots x_{j_m} \quad \text{and} \quad b_j \left( x^{(1)}, \ldots , x^{(m)}\right)=x_{j_1}^{(1)} \ldots x_{j_m}^{(m)}.\]
 If we restate \eqref{eq3} with this notation we get \eqref{eq4}, with $c_1=1$ and $c_2=C(m,n)$.
 Therefore, we deduce from {Pe{\l}czy\'nski's} theorem that
 \begin{align}
 \label{eq8}
 \frac{1}{C(m,n)} \int_{\T^n} \left\| \sum_{j \in J} v_j x_{j_1} \ldots x_{j_m} \right\|_E \, dx 
 &\leq \int_{\T^n}\ldots \int_{\T^n} \left\| \sum_{j \in J} v_j x_{j_1}^{(1)} \ldots x_{j_m}^{(m)} \right\|_E \, dx^{(1)} \ldots  dx^{(m)} \notag
 \\& \leq C(m,n) \int_{\T^n} \left\| \sum_{j \in J} v_j x_{j_1} \ldots x_{j_m} \right\|_E \, dx.
 \end{align}
 for every {Banach} space $E$ and every sequence of vectors $(v_j)_{j \in J}\subseteq E$.
 Choosing the space $E$ and the vectors $(v_j)_{j \in J}\subseteq E$ adequately will yield the estimate we seek.
 
 We will build upon an example provided by {Bourgain} (unpublished) and included in a paper by {McConnell} and {Taqqu} \cite[Example 4.1]{mctaq} (see also \cite[Section~6.9]{kwapien}).
 Consider the {Banach} space $F= \mathcal{L}(\ell_2)$. For every $1\leq i \neq j \leq n$, define vectors $v_{ij}\in F$ by
 \[v_{ij} = \frac{1}{i-j} e_i\otimes e_j+\frac{1}{j-i} e_j\otimes e_i.\]
 Using complex {Steinhaus} variables instead of {Bernoulli} random variables and proceeding as in \cite{kwapien} we get
 \begin{align}
 \label{eq6}
  \int_{\T^n}\left\|\sum_{1\leq i<j \leq n} v_{ij}x_i x_j\right\|&  \, dx \leq \pi \quad \quad \text{and}   
  \\ \label{eq7} \int_{\T^n} \int_{\T^n} \left\|\sum_{1\leq i<j \leq n} v_{ij}x_i^{(1)} x_j^{(2)}\right\|& \, dx^{(1)} dx^{(2)} \geq \log n -\pi.
 \end{align}
 
 Note that by the previous estimations we obtain the desired result for $m=2$ since we have
 \begin{multline*}
  \log n -\pi\leq \int_{\T^n} \int_{\T^n} \left\|\sum_{1\leq i<j \leq n} v_{ij}x_i^{(1)} x_j^{(2)}\right\| \, dx^{(1)} dx^{(2)}
 \\ \leq C(2,n) \int_{\T^n}\left\|\sum_{1\leq i<j \leq n} v_{ij}x_i x_j\right\|  \, dx \leq C(2,n) \pi.
 \end{multline*}
 Moreover, this together with Theorem \ref{teo2} shows that the asymptotic behaviour of $C(2,n)$ is logarithmic.
 
 To conclude our argument it remains to extend this 2-variable example to $m$ variables.
 Assume $m$ is even and let $E= \bigotimes_{k=1}^{m/2}F$ be the projective tensor product of $m/2$ copies of $F$.
 Consider the $m$-homogeneous vector-valued polynomial $P:\mathbb{C}^n\rightarrow E$ defined by
\[P(x)= \sum_{\substack{\frac{2n}{m}(k-1)<j_{2k-1}<j_{2k}\leq \frac{2n}{m}k\\1\leq k \leq \frac{m}{2}}}v_jx_j, 
\quad \text{where } v_j= v_{j_1 j_2}\otimes v_{j_3 j_4}\otimes \ldots \otimes v_{j_{m-1} j_m}.\]
Notice that 
\[P(x)= \otimes_{k=1}^{m/2} \sum_{\frac{2n}{m}(k-1)<j_{2k-1}<j_{2k}\leq \frac{2n}{m}k}v_{j_{2k-1}j_{2k}}x_{j_{2k-1}}x_{j_{2k}}.\]
Applying \eqref{eq6} we get
\begin{align*}
 \int_{\T^n}\left\|P(x)\right\| \, dx &=
 \int_{\T^n} \prod_{k=1}^{m/2} \left\|\sum_{\frac{2n}{m}(k-1)<j_{2k-1}<j_{2k}\leq \frac{2n}{m}k}v_{j_{2k-1}j_{2k}}x_{j_{2k-1}}x_{j_{2k}}\right\| \, dx
 \\ &= \prod_{k=1}^{m/2} \int_{\T^n} \left\|\sum_{\frac{2n}{m}(k-1)<j_{2k-1}<j_{2k}\leq \frac{2n}{m}k}v_{j_{2k-1}j_{2k}}x_{j_{2k-1}}x_{j_{2k}}\right\| \, dx
 \\ &\leq \prod_{k=1}^{m/2} \pi = \pi^{m/2}.
 \end{align*}
 On the other hand, from \eqref{eq7} we deduce
 \begin{align*}
 \int_{\T^n} \ldots \int_{\T^n}&\left\|L_P\left(x^{(1)},\ldots,x^{(m)}\right)\right\| \, dx^{(1)} \ldots  dx^{(m)}=
 \\ &=\int_{\T^n} \ldots \int_{\T^n} \prod_{k=1}^{m/2} 
 \left\|\sum_{\frac{2n}{m}(k-1)<j_{2k-1}<j_{2k}\leq \frac{2n}{m}k}v_{j_{2k-1}j_{2k}}x_{j_{2k-1}}^{(2k-1)}x_{j_{2k}}^{(2k)}\right\| \, dx^{(1)} \ldots  dx^{(m)}
 \\ &= \prod_{k=1}^{m/2} \int_{\T^n} \int_{\T^n} 
 \left\|\sum_{\frac{2n}{m}(k-1)<j_{2k-1}<j_{2k}\leq \frac{2n}{m}k}v_{j_{2k-1}j_{2k}}x_{j_{2k-1}}^{(2k-1)}x_{j_{2k}}^{(2k)}\right\| \, dx^{(2k-1)} dx^{(2k)}
 \\ &\geq \prod_{k=1}^{m/2} \left(\log\left(\frac{2n}{m}\right)-\pi\right) =  \left(\log\left(\frac{2n}{m}\right)-\pi\right)^{m/2}. 
\end{align*}
Finally, using \eqref{eq8} together with these estimates we obtain
\[C(n,m)\geq \left(\frac{\log\left(\frac{2n}{m}\right)-\pi}{\pi}\right)^{m/2},\]
as desired.
\end{proof}

Note that Lemmas \ref{lema3} and \ref{lema4} together with Theorem \ref{teo2} prove Theorem \ref{teo1}.

\begin{remark}
Tracing back the argument to obtain \eqref{eq7}, we find that Bourgain's example is based on a lower estimate of the main triangle projection's norm on $\mathcal{L}(\ell_2)$.
In other words, the lower bound for $C(m,n)$ was obtained by studying the behaviour of the main triangle projection as mentioned in Remark \ref{rem1}.
Although $C(m,n)$ and $C_p(m,n)$ were not completely characterized, it seems that the main triangle projection plays a crucial role in determining their asymptotic behaviour.
\end{remark}

\section*{Acknowledgements}
I thank my supervisor Daniel Carando for his guidance and fruitful discussions and  Sunke Schl\"uters for his helpful comments.



\begin{thebibliography}{10}
 \bibitem{ben}
 Bennett, G. ``Unconditional convergence and almost everywhere convergence''. In: \textit{Zeitschrift f{\"u}r Wahrscheinlichkeitstheorie und Verwandte Gebiete} 34.2 (1976), pp. 135-155.
 	
 \bibitem{nosim}
 Defant, A. and Schl\"uters, S. ``Non-symmetric polarization''. In: \textit{Journal of Mathematical Analysis and Applications} 445.2 (2017), pp. 1291-1299.
 
 \bibitem{din}
 Dineen, S. \textit{Complex analysis on infinite dimensional spaces}. London: Springer-Verlag, 1999.
 
 
 \bibitem{fiya}
 Fisher, R. A. and Yates, F. \textit{Statistical tables for biological, agricultural and medical research}. Edinburgh: Oliver and Boyd, 1938, p. 285.
 

 \bibitem{kwap}
 Kwapie\'n, S. and Pe{\l}czy\'nski, A. ``The main triangle projection in matrix spaces and its applications''. eng. In: \textit{Studia Mathematica} 34.1 (1970), pp. 43-67.
 
 \bibitem{kwapien}
 Kwapien, S. and Woyczynski, W. \textit{Random series and stochastic integrals: single and multiple}. Birkh\"auser Basel, 1992.
 
 \bibitem{mctaq}
 McConnell, T. R. and Taqqu, M. S. ``Decoupling of Banach-valued multilinear forms in independent symmetric Banach-valued random variables''. In: \textit{Probability Theory and Related Fields} 75.4 (1987), pp. 499-507.
 
 \bibitem{pel}
 Pe{\l}czy\'nski, A. ``Commensurate Sequences of Characters''. In: \textit{Proceedings of the American Mathematical Society} 104.2 (1988), pp. 525-531.
 \end{thebibliography}
\end{document}